\documentclass[11pt]{amsart}
\usepackage[margin=1in]{geometry}

\usepackage{graphicx}
\usepackage[active]{srcltx}
\usepackage{amsmath}
\usepackage{amscd}
\usepackage{tikz-cd}
\usepackage{amssymb}
\usepackage[mathscr]{euscript}
\usepackage{enumerate}
\usepackage{hyperref}

\DeclareMathOperator{\Hom}{Hom}
\DeclareMathOperator{\Soc}{Soc}
\DeclareMathOperator{\Ann}{Ann}
\DeclareMathOperator{\Ext}{Ext}

\DeclareMathOperator{\gr}{gr}

\DeclareMathOperator{\DD}{D}

\DeclareMathOperator{\Ass}{Ass}

\newcommand{\NN}{\mathbb{N}}
\newcommand{\D}{\mathscr{D}}
\newcommand{\F}{\mathscr{F}}

\DeclareMathOperator{\grHom}{\sideset{^*}{}\Hom}
\DeclareMathOperator{\grExt}{\sideset{^*}{}\Ext}

\DeclareMathOperator{\E}{E}

\newcommand{\fm}{\mathfrak{m}}
\newcommand{\fp}{\mathfrak{p}}
\newcommand{\fq}{\mathfrak{q}}

\theoremstyle{plain}
\newtheorem{thm}{Theorem}[section]
\newtheorem*{thm*}{Theorem}
\newtheorem{bigthm}{Theorem}
\newtheorem{prop}[thm]{Proposition}
\newtheorem{lem}[thm]{Lemma}

\theoremstyle{definition}
\newtheorem{definition}[thm]{Definition}
\newtheorem{propdefn}[thm]{Proposition-Definition}

\theoremstyle{remark}
\newtheorem{remark}[thm]{Remark}
\newtheorem{example}[thm]{Example}
\newtheorem{convention}[thm]{Convention}

\numberwithin{equation}{thm}

\renewcommand{\E}{E}

\begin{document}

\title{Duality and de Rham cohomology for graded $\D$-modules}
\author{Nicholas Switala \and Wenliang Zhang}
\address{Department of Mathematics, Statistics, and Computer Science \\ University of Illinois at Chicago \\ 322 SEO (M/C 249) \\ 851 S. Morgan Street \\ Chicago, IL 60607}
\email{nswitala@uic.edu, wlzhang@uic.edu}
\thanks{The first author gratefully acknowledges NSF support through grant DMS-1604503. The second author is partially supported by the NSF through grant DMS-1606414.}

\begin{abstract}
We consider the (graded) Matlis dual $\DD(M)$ of a graded $\D$-module $M$ over the polynomial ring $R = k[x_1, \ldots, x_n]$ ($k$ is a field of characteristic zero), and show that it can be given a structure of $\D$-module in such a way that, whenever $\dim_kH^i_{dR}(M)$ is finite, then $H^i_{dR}(M)$ is $k$-dual to $H^{n-i}_{dR}(\DD(M))$. As a consequence, we show that if $M$ is a graded $\D$-module such that $H^n_{dR}(M)$ is a finite-dimensional $k$-space, then 
$\dim_k(H^n_{dR}(M))$ is the maximal integer $s$ for which there exists a surjective $\D$-linear homomorphism $M \rightarrow E^s$, where $E$ is the top local cohomology module $H^n_{(x_1, \ldots, x_n)}(R)$. This extends a recent result of Hartshorne and Polini on formal power series rings to the case of polynomial rings; we also apply the same circle of ideas to provide an alternate proof of their result. When $M$ is a finitely generated graded $\D$-module such that $\dim_kH^i_{dR}(M)$ is finite, we generalize the above result further, showing that $H^{n-i}_{dR}(M)$ is $k$-dual to $\Ext_{\D}^i(M, \E)$.
\end{abstract}

\maketitle

\section{Introduction}\label{intro}

Let $k$ be a field of characteristic zero and let $R$ denote either a polynomial ring or formal power series ring in $n$ variables over $k$. Let $\D$ be the ring of $k$-linear differential operators on $R$, {\it i.e.}, $\D=R\langle \frac{\partial}{\partial x_1},\dots,\frac{\partial}{\partial x_n}\rangle$. Each (left) $\D$-module $M$ admits a de Rham complex $M \otimes \Omega^{\bullet}_R$, where $\Omega^{\bullet}_R$ is the canonical de Rham complex of $R$. The cohomology spaces of $M \otimes \Omega^{\bullet}_R$ are called the de Rham cohomology spaces of $M$ and denoted by $H^i_{dR}(M)$. The study of de Rham cohomology of $\D$-modules has a long and rich history, and has found numerous applications in different areas of mathematics. 

When $R=k[[x_1,\dots,x_n]]$, in a recent paper \cite{hartpol}, Hartshorne and Polini investigate the connection between the top de Rham cohomology $H^n_{dR}(M)$ of a \emph{holonomic} $\D$-module $M$ and $\D$-linear maps $M\to E^s$ where $E$ is the local cohomology module $H^n_{(x_1, \ldots, x_n)}(R)$. We rephrase their result as follows:

\begin{thm}[{\it cf.} Theorem 5.1 in \cite{{hartpol}}]
\label{dRdim and surjective map}
Let $M$ be a holonomic $\D$-module. Then
\[\dim_k(H^n_{dR}(M))=\max\{s\in\mathbb{N}\mid \exists\ {\rm a\ surjective\ }\varphi\in \Hom_{\D}(M,E^s) \}.\]
\end{thm}

Theorem \ref{dRdim and surjective map} has the following striking consequences when applied to certain local cohomology modules:

\begin{thm}\cite[Theorem 6.4 and Corollary 6.6]{hartpol}\label{hartmainthm}
Let $V \subseteq \mathbb{P}^n_k$ be a nonsingular projective variety of codimension $c$, where $k$ is a field of characteristic zero. Let $I \subseteq R = k[x_0, \ldots, x_n]$ be the homogeneous defining ideal of $V$. Then there is a simple $\D$-submodule $M \subseteq H^c_I(R)$, supported on the affine cone $C(V)$ over $V$, such that the quotient $H^c_I(R)/M\cong E^{b_d - b_{d-2}}$ where $d = \dim(V)$ and $b_i = \dim_k H_i^{dR}(V)$ are the dimensions of the algebraic de Rham cohomology spaces of $V$.

In particular, if $I \subseteq R = k[x_0, \ldots, x_n]$ is the homogeneous defining ideal of an embedding of $\mathbb{P}^d_k$ into $\mathbb{P}^n_k$, then the local cohomology module $M = H^{n-d}_I(R)$ is simple as a $\D$-module.
\end{thm}

Hartshorne and Polini observe that Theorem \ref{dRdim and surjective map} fails in general over polynomial rings, so in order to establish Theorem \ref{hartmainthm}, they pass to completions. 

In this paper, we investigate extensions of Theorem \ref{dRdim and surjective map} to \emph{graded} $\D$-modules over polynomial rings. 
To this end, we develop a theory of graded Matlis duality for graded $\D$-modules. One of our main results is that the graded Matlis dual is compatible with de Rham cohomology in the case that this cohomology is finite-dimensional:

\begin{bigthm}[Proposition \ref{gradedDdual} and Theorem \ref{dRdualsg}]
\label{gradedduality}
Let $R = k[x_1, \ldots, x_n]$ where $k$ is a field of characteristic zero, and let $\D = \D(R,k)$ be the ring of $k$-linear differential operators on $R$. Let $M$ be a graded $\D$-module (Definition \ref{gradeddmod}). The graded Matlis dual $\DD(M)$ (Definition \ref{shifteddual}) has a natural structure of graded $\D$-module. For all $i$ such that the de Rham cohomology space $H^i_{dR}(M)$ of $M$ is finite-dimensional over $k$, we have 
\[H^i_{dR}(M)^{\vee} \cong H^{n-i}_{dR}(\DD(M))\] 
as $k$-spaces, where ${}^{\vee}$ denotes $k$-space dual.
\end{bigthm}

We apply Theorem \ref{gradedduality} to prove the following extension of Theorem \ref{dRdim and surjective map} to the case of polynomial rings:

\begin{bigthm}[Theorem \ref{gradedhart}]\label{thmb}
Let $R$ and $\D$ be as in Theorem \ref{gradedduality}, and let $M$ be a graded $\D$-module such that $H^n_{dR}(M)$ is a finite-dimensional $k$-space. Then 
\[\dim_k(H^n_{dR}(M))=\max\{s \in \mathbb{N}\mid \exists\ {\rm a\ surjective\ }\varphi\in \Hom_{\D}(M,\E^s) \},\]
where $\E$ is the top local cohomology module $H^n_{(x_1, \ldots, x_n)}(R)$.
\end{bigthm}

We observe that the local cohomology module $M = H^c_I(R)$ in the statement of Theorem \ref{hartmainthm} is a graded $\D$-module (since $I \subseteq R$ is a homogeneous ideal), and its de Rham cohomology spaces are finite-dimensional (since it is a \emph{holonomic} $\D$-module). Therefore Theorem \ref{thmb} applies to $M$ and can be used to recover Theorem \ref{hartmainthm}. The idea of using duality to prove Theorem \ref{thmb} also leads us to an alternate proof of Theorem \ref{dRdim and surjective map} in the formal power series case.


We also give the following interpretation of the other de Rham cohomology spaces of $M$, a graded analogue of a recent result of Lyubeznik \cite[Theorem 1.3]{gennadyext} in the formal power series case.

\begin{bigthm}[Theorem \ref{gradedgennady}]
\label{de Rham and ext}
Let $R$ and $\D$ be as in Theorem \ref{gradedduality}. Let $M$ be a finitely generated graded left $\D$-module. For all $i$ such that $H^i_{dR}(M)$ is a finite-dimensional $k$-space,  
\[H^{n-i}_{dR}(M)^{\vee} \cong \Ext_{\D}^i(M, \E)\] 
as $k$-spaces. 
\end{bigthm}


We note that if $M$ is a graded \emph{holonomic} $\D$-module (\emph{e.g.} a local cohomology module of $R$ supported in a homogeneous ideal), then the hypotheses of Theorems \ref{gradedduality}, \ref{thmb}, and \ref{de Rham and ext} are satisfied.

This paper is organized as follows: in section \ref{prelims}, we summarize background material on $\D$-modules, de Rham cohomology, Matlis duality for $\D$-modules in the formal power series case (following \cite{swithesis}), generalities on graded modules, and graded Matlis duality in the polynomial case; in section \ref{graded}, we describe the graded $\D$-module structure on the graded Matlis dual of a graded $\D$-module and proceed to the proofs of our main results; the brief section \ref{hartpolalt} contains our alternate proof of Hartshorne and Polini's result (Theorem \ref{dRdim and surjective map}) in the formal power series case; section \ref{exttor} is denoted to proving Theorem \ref{de Rham and ext}; finally, in section \ref{noninjective}, we apply our Theorem \ref{thmb} to the specific $\D$-module $\E$ to show that $\E$ is not an injective object in the category of graded holonomic $\D$-modules, which shows that the analogue for holonomic $\D$-modules of \cite[Corollary 2.10]{lsw} does not hold. 

\subsection*{Acknowledgments} The authors thank Linquan Ma and Uli Walther for helpful discussions, and thank Robin Hartshorne, Gennady Lyubeznik, Claudia Polini, and Amnon Yekutieli for comments on an earlier draft.
\section{Preliminaries}\label{prelims}
In this section, we collect some background material on $\D$-modules (over polynomial and formal power series rings), review the theory of Matlis duality for $\D$-modules in the formal power series case developed in \cite{swithesis}, and recall the graded version of Matlis duality.

Throughout this paper, $k$ is a field of characteristic zero, $R$ is either the polynomial ring $k[x_1, \ldots, x_n]$ or the formal power series ring $k[[x_1, \ldots, x_n]]$, and $\mathfrak{m}$ is the maximal ideal $(x_1, \ldots, x_n) \subseteq R$. We let $E = H^n_{\mathfrak{m}}(R)$ be the top local cohomology module of $R$ supported in $\mathfrak{m}$, which is an injective hull of $k = R/\mathfrak{m}$ in the category of $R$-modules. The underlying $k$-space of $E$ is the same whether $R$ is the polynomial ring or the formal power series ring: it is spanned by \emph{inverse monomials} $x_1^{i_1}\cdots x_n^{i_n}$ where all $i_j \leq -1$. Finally, ${}^{\vee}$ always denotes $k$-space dual.

\subsection{$\D$-modules and de Rham cohomology}\label{derham} 

In this subsection, $R$ is either the polynomial ring $k[x_1, \ldots, x_n]$ or the formal power series ring $k[[x_1, \ldots, x_n]]$. Our basic reference for the material in this subsection is \cite{bjork}. 

We denote by $\D$ the non-commutative ring $\D(R,k)$ of $k$-linear differential operators on $R$. As an $R$-module, $\D$ is free on the monomials $\partial_1^{i_1}\cdots\partial_n^{i_n}$ where $i_1, \ldots, i_n \geq 0$ (here $\partial_i$ denotes the partial differentiation operator $\frac{\partial}{\partial x_i}: R \rightarrow R$); as a ring, $\D = R\langle \partial_1, \ldots, \partial_n\rangle$ with the relations $\partial_i\partial_j = \partial_j\partial_i$ and $\partial_if = \partial_i(f) + f\partial_i$ for all $i$ and $j$ and all $f \in R$. A $\D$-module $M$ is an $R$-module together with a \emph{left} action of $\D$ on $M$ (when we need to consider a \emph{right} $\D$-module, we will say so explicitly).

The ring $\D$ has an increasing filtration $\{F^l\D\}$, called the \emph{order} (or \emph{degree}) filtration, where $F^l\D$ consists of those differential operators in which each term has no more than $l$ partial derivatives. The associated graded object $\gr(\D) = \oplus_l F^l\D/F^{l-1}\D$ with respect to this filtration is isomorphic to $R[\xi_1, \ldots, \xi_n]$ (a commutative ring), where $\xi_i$ is the image of $\partial_i$ in $F^1\D/F^0\D \subseteq \gr(\D)$. If $M$ is a finitely generated left $\D$-module, there exists a \emph{good filtration} $\{G^pM\}$ on $M$, meaning that $M$ becomes a filtered left $\D$-module with respect to the order filtration on $\D$ \emph{and} $\gr(M) = \oplus_p G^pM/G^{p-1}M$ is a finitely generated $\gr(\D)$-module.  We let $J$ be the radical of $\Ann_{\gr(\D)} \gr(M) \subseteq \gr(\D)$ and set $d(M) = \dim \gr(\D)/J$ (Krull dimension).  The ideal $J$, and hence the number $d(M)$, is independent of the choice of good filtration on $M$.  

By \emph{Bernstein's theorem}, if $M \neq 0$ is a finitely generated left $\D$-module, we have $n \leq d(M) \leq 2n$.  In the case $d(M) = n$ we say that $M$ is \emph{holonomic}.  It is known (see \cite[\S 1.5, 3.3]{bjork}) that submodules and quotients of holonomic $\D$-modules are holonomic, an extension of a holonomic $\D$-module by another holonomic $\D$-module is holonomic, holonomic $\D$-modules are of finite length over $\D$, and holonomic $\D$-modules are cyclic (generated over $\D$ by a single element). Examples of holonomic $\D$-modules include $R$ itself, $E = H^n_{\mathfrak{m}}(R)$, and more generally any \emph{local cohomology} module $H^i_I(M)$ where $I \subseteq R$ is an ideal and $M$ is a holonomic $\D$-module (see \cite{brodmann} for generalities concerning local cohomology). By Kashiwara's equivalence \cite[Example 1.6.4]{hotta}, if $M$ is a holonomic $\D$-module whose support as an $R$-module consists only of the maximal ideal $\mathfrak{m}$, then $M$ is a finite direct sum of copies of $E$.

Given any $\D$-module $M$, we can define its \emph{de Rham complex}.  This is a complex of length $n$, denoted $M \otimes \Omega_R^{\bullet}$ (or simply $\Omega_R^{\bullet}$ in the case $M = R$), whose objects are $R$-modules but whose differentials are merely $k$-linear.  It is defined as follows \cite[\S 1.6]{bjork}: for $0 \leq i \leq n$, $M \otimes \Omega^i_R$ is a direct sum of $n \choose i$ copies of $M$, indexed by $i$-tuples $1 \leq j_1 < \cdots < j_i \leq n$.  The summand corresponding to such an $i$-tuple will be written $M \, dx_{j_1} \wedge \cdots \wedge dx_{j_i}$. The $k$-linear differentials $d^i: M \otimes \Omega_R^i \rightarrow M \otimes \Omega_R^{i+1}$ are defined by 
\[
d^i(m \,dx_{j_1} \wedge \cdots \wedge dx_{j_i}) = \sum_{s=1}^n \partial_s(m)\, dx_s \wedge dx_{j_1} \wedge \cdots \wedge dx_{j_i},
\]
with the usual exterior algebra conventions for rearranging the wedge terms, and extended by linearity to the direct sum. We remark that in the polynomial case, we are simply using the usual K\"{a}hler differentials to build this complex, whereas in the formal power series case, we are using the $\mathfrak{m}$-adically continuous differentials (since in this case the usual module $\Omega^1_{R/k}$ of K\"{a}hler differentials is not finitely generated over $R$). The cohomology objects $h^i(M \otimes \Omega_R^{\bullet})$, which are $k$-spaces, are called the \emph{de Rham cohomology spaces} of the left $\D$-module $M$, and are denoted $H^i_{dR}(M)$. The simplest de Rham cohomology spaces (the $0$th and $n$th) of $M$ take the form
\[
H^0_{dR}(M) = \{m \in M \mid \partial_1(m) = \cdots = \partial_n(m) = 0\} \subseteq M; \quad H^n_{dR}(M) = M/(\partial_1(M) + \cdots + \partial_n(M)).
\]

The following theorem is standard (see \cite[Theorem 1.6.1]{bjork}) in the polynomial case, and is due to van den Essen \cite[Proposition 2.2]{vdEcoker2} in the more difficult formal power series case:

\begin{thm}\label{dRfindim}
Let $M$ be a holonomic $\D$-module. The de Rham cohomology spaces $H^i_{dR}(M)$ are finite-dimensional over $k$ for all $i$.
\end{thm} 

\begin{example}\label{poincare}
The de Rham cohomology of $R$ itself is $k$ in degree $0$ and $0$ otherwise; this is the ``algebraic Poincar\'{e} lemma'', proved in the polynomial case in \cite[Proposition II.7.1]{HartDR} (the same proof works in the formal power series case). The de Rham cohomology of $E$ is $k$ in degree $n$ and $0$ otherwise \cite[Example 2.2(4)]{hartpol}.
\end{example}

The dimension of the $0$th de Rham cohomology space of a $\D$-module $M$ has the following useful interpretation:  

\begin{lem}\cite[Example 2.2(6)]{hartpol}\label{drzero}
Let $M$ be a $\D$-module such that $H^0_{dR}(M)$ is a finite-dimensional $k$-space. Then the $k$-dimension of $H^0_{dR}(M)$ is equal to the maximal integer $s$ for which there exists an injective $\D$-module homomorphism $R^s \rightarrow M$.
\end{lem}

We remark that in \cite{hartpol}, $M$ is assumed to be a holonomic $\D$-module in the statement of the preceding lemma. However, all that is needed for the proof is the finite-dimensionality of $H^0_{dR}(M)$, and we will need this stronger statement below. In the proof of Lemma \ref{drzero} (in the formal power series case), we will need the following result of van den Essen:

\begin{lem}\cite[Lemme 1]{vdEkernel}\label{kerneldep}
Suppose that $R$ is the formal power series ring $k[[x_1, \ldots, x_n]]$. Let $M$ be a $\D$-module, and denote by $M_*$ the kernel of $\partial_n: M \rightarrow M$. Any $R$-linear dependence relation among elements of $M_*$ holds homogeneously in $x_n$: that is, if $f_1, \ldots, f_l \in R$ and $m_1, \ldots, m_l \in M_*$ are such that $f_1m_1 + \cdots + f_lm_l = 0$, then we have $f_{1,j}m_1 + \cdots + f_{l,j}m_l = 0$ for all $j$, where $f_{i,j} \in k[[x_1, \ldots, x_{n-1}]]$ is the coefficient of $x_n^j$ in $f_i$.
\end{lem}

\begin{proof}[Proof of Lemma \ref{drzero}]
Let $\{m_1, \ldots, m_t\}$ be a $k$-basis for $H^0_{dR}(M)$. By the definition of the de Rham complex, we have $\partial_i(m_j) = 0$ for $1 \leq i \leq n$ and $1 \leq j \leq t$. Define a map $\lambda: R^t \rightarrow M$ by 
\[
\lambda(r_1, \ldots, r_t) = r_1m_1 + \cdots + r_tm_t,
\]
which, since $r_i \in R \subseteq \D$ for all $i$, is clearly $\D$-linear. We claim that $\lambda$ is injective. Suppose not, and let $r_1, \ldots, r_t$ be elements of $R$ (not all zero) such that $r_1m_1 + \cdots + r_tm_t = 0$. Observe that for all $i$ and $j$, we have $\partial_i(r_jm_j) = \partial_i(r_j)m_j + r_j\partial_i(m_j) = \partial_i(r_j)m_j$ (since $\partial_i(m_j) = 0$), and consequently
\[
0 = \partial_i(0) = \partial_i(r_1m_1 + \cdots + r_tm_t) = \partial_i(r_1)m_1 + \cdots + \partial_i(r_t)m_t.
\]

At this point we must treat the polynomial and formal power series cases separately. If $R = k[x_1, \ldots, x_n]$, then it is clear from the displayed equality that we can simply differentiate repeatedly until all nonzero coefficients are scalars, contradicting the $k$-linear independence of the $m_i$. 

On the other hand, in the case $R = k[[x_1, \ldots, x_n]]$, we may similarly differentiate the given $R$-linear dependence relation repeatedly to obtain a new $R$-linear dependence relation in which at least one coefficient is a unit. By Lemma \ref{kerneldep}, any $R$-linear dependence relation among elements in $\ker(\partial_n)$ (in particular, among elements of $H^0_{dR}(M)$) holds homogeneously in $x_n$; taking the $x_n^0$-term, we obtain an $R_{n-1}$-linear dependence relation among $m_1, \ldots, m_t$. Applying Lemma \ref{kerneldep} $n-1$ more times, we obtain a $k$-linear dependence relation among $m_1, \ldots, m_t$: to be specific, we find that $r_{1,0}m_1 + \cdots + r_{t,0}m_t = 0$ where $r_{i,0}$ is the constant term of $r_i$. By assumption, at least one of these constant terms is nonzero, so the $k$-linear dependence relation is nontrivial, contradicting the fact that $\{m_1, \ldots, m_t\}$ is a $k$-basis of $H^0_{dR}(M)$. We conclude that in either the polynomial or formal power series case, we have $H^0_{dR}(M) = t \leq s$. 

The converse inequality is easier: if $R^s \rightarrow M$ is an injective $\D$-linear homomorphism, it restricts to a injective $k$-linear map $H^0_{dR}(R^s) \rightarrow H^0_{dR}(M)$, and since $\dim_k H^0_{dR}(R^s) = s$ by Example \ref{poincare}, we have $s \leq t$ as well, completing the proof.
\end{proof}   

\subsection{Matlis duality for $\D$-modules}\label{thesis}

In this subsection, $R = k[[x_1, \ldots, x_n]]$ is the formal power series ring and $\fm=(x_1,\dots,x_n)$. This subsection summarizes some of the theory in \cite{swithesis}. See \cite[\S 18]{matsumura} for proofs of the basic facts about Matlis duality (over any complete local ring) that appear in the following paragraph.

Recall that the \emph{Matlis dual} of an $R$-module $M$ is the $R$-module $D(M) = \Hom_R(M,E)$ where $E=H^n_{\fm}(R)$. In particular, we have $D(R) = E$ and $D(E) = R$. The contravariant functor $D$ is exact and defines an anti-equivalence between the category of finitely generated $R$-modules and the category of Artinian $R$-modules. If $M$ is finitely generated or Artinian, the canonical evaluation map 
\[
\iota_M: M \rightarrow D(D(M)) = \Hom_R(\Hom_R(M,E),E)
\]
is an isomorphism of $R$-modules. More generally, $\iota_M$ is an isomorphism if and only if $M/N$ is Artinian for some finitely generated $R$-submodule $N \subseteq M$ \cite[Proposition 1.3]{enochs}. (Such modules are called \emph{Matlis reflexive}.)

Let $\sigma: E \rightarrow k$ be the \emph{residue map}, that is, the $k$-linear projection of $E \cong \oplus_{i_1, \ldots, i_n > 0} k \cdot x_1^{-i_1}\cdots x_n^{-i_n}$ onto its $x_1^{-1}\cdots x_n^{-1}$-component. This component is the \emph{socle} $\Soc(E) = (0 :_E \mathfrak{m})$ of $E$. (Any projection of $E$ onto its socle will suffice for our purposes; we make this choice for concreteness.) If $M$ is an $R$-module, post-composition with $\sigma$ defines an injective homomorphism of $R$-modules
\[
\Phi_M: D(M) = \Hom_R(M,E) \rightarrow \Hom_k(M,k)
\]
whose image consists of precisely those $k$-linear maps $\lambda: M \rightarrow k$ that are $\mathfrak{m}$-adically continuous when restricted to any finitely generated $R$-submodule $N \subseteq M$. Such maps are called \emph{$\Sigma$-continuous} in \cite{swithesis} or \emph{continuous} in \cite{hartpol}. We summarize the above in the following proposition, which is stated without proof in \cite[Remarque IV.5.5]{sga2}, and proved in detail in \cite[Theorem 3.15]{swithesis} (see also \cite[Proposition 5.4]{hartpol}; in all these references, the result is stated more generally for a complete local ring with a coefficient field):

\begin{propdefn}\label{sigmadual}
Let $M$ be an $R$-module. We say that a $k$-linear map $\lambda: M \rightarrow k$ is \emph{$\Sigma$-continuous} if for every finitely generated $R$-submodule $N \subseteq M$, there exists an integer $l$ such that $\lambda(\mathfrak{m}^lN) = 0$. We denote the set (indeed, $R$-module) of $\Sigma$-continuous maps $M \rightarrow k$ by $D^{\Sigma}(M)$ and refer to it as the \emph{$\Sigma$-continuous dual} of $M$. There is an isomorphism of $R$-modules $\Phi_M: D(M) \rightarrow D^{\Sigma}(M)$ defined by post-composition with the residue map $\sigma: E \rightarrow k$ and functorial in $M$.
\end{propdefn}

Note that if $M$ is finitely generated, $D^{\Sigma}(M)$ is the continuous $k$-dual of $M$, and if $M$ is Artinian (so that every finitely generated submodule of $M$ is of finite length), $D^{\Sigma}(M)$ is simply the $k$-dual of $M$.

Now suppose that $M$ is a $\D$-module. By using the identification of Proposition-Definition \ref{sigmadual}, we can endow the Matlis dual $D(M)$ with a structure of $\D$-module, as follows. Given a differential operator $\delta \in \D$, we write $\delta_M: M \rightarrow M$ for its action on $M$. If $\lambda: M \rightarrow k$ is a $\Sigma$-continuous map, so also is $\lambda \circ \delta_M: M \rightarrow k$ \cite[Proposition 4.8]{swithesis}. By setting $\lambda \cdot \delta = \lambda \circ \delta_M$, we obtain a structure of right $\D$-module on $D^{\Sigma}(M)$, and by transport of structure, $D(M)$ becomes a right $\D$-module as well. There is a simple \emph{transposition} operation that converts right $\D$-modules to left $\D$-modules (with the same underlying $R$-module) and conversely (we will explain this operation in more detail below in the polynomial case: see Definition \ref{transpose}). After transposing, we get a (left) $\D$-module structure on the Matlis dual $D(M)$ of a (left) $\D$-module $M$. 

\begin{lem}\label{dualofdlinear}\cite[Proposition 4.11]{swithesis}
Let $M$ and $N$ be $\D$-modules, and let $\varphi: M \rightarrow N$ be a $\D$-linear map. The Matlis dual $\varphi^*$ (that is, the map $D^{\Sigma}(N) \rightarrow D^{\Sigma}(M)$ defined by pre-composition with $\varphi$) is $\D$-linear as well.
\end{lem}

\begin{proof}
We work with the right $\D$-module structures; the result remains true, of course, after transposing. Let $\delta \in \D$ be given. Since $\varphi$ is $\D$-linear, we have $\varphi \circ \delta_M = \delta_N \circ \varphi$. Therefore, if $\lambda \in D^{\Sigma}(N)$, we have
\[
\varphi^*(\lambda \cdot \delta) = \varphi^*(\lambda \circ \delta_N) = \lambda \circ \delta_N \circ \varphi = \lambda \circ \varphi \circ \delta_M = \varphi^*(\lambda) \circ \delta_M = \varphi^*(\lambda) \cdot \delta,
\]
so that $\varphi^*$ is $\D$-linear.
\end{proof}

\begin{lem}\label{evaldlinear}\cite[Proposition 4.12]{swithesis}
Let $M$ be a $\D$-module. The canonical evaluation map $\iota_M: M \rightarrow D^{\Sigma}(D^{\Sigma}(M))$ is $\D$-linear.
\end{lem}

\begin{proof}
Let $m \in M$ and $\delta \in \D$ be given. Since $\D$ acts on $D^{\Sigma}(D^{\Sigma}(M))$ by (iterated) pre-composition, $\delta \cdot \iota_M(m)$ is the map $D^{\Sigma}(M) \rightarrow M$ defined by evaluation at $\delta \cdot m$, which is exactly $\iota_M(\delta \cdot m)$.
\end{proof}

Finally, we have the following theorem on the de Rham cohomology of Matlis duals:

\begin{thm}\cite[Theorem 5.1]{swithesis}\label{dRduals}
If $M$ is a holonomic $\D$-module, then 
\[
H^i_{dR}(M)^{\vee} \cong H^{n-i}_{dR}(D(M)) 
\]
as $k$-spaces for all $i$.
\end{thm}

\begin{example}\label{hellus}
Even if $M$ is holonomic, the Matlis dual $D(M)$ need \emph{not} be holonomic. For example, it follows from a result of Hellus \cite[Theorem 2.4]{hellus} that if $R = k[[x_1, \ldots, x_n]]$ with $n \geq 2$ and $M$ is the local cohomology module $H^1_{(x_1)}(R)$, then every prime ideal of $R$ that does not contain $x_1$ is an associated prime of the Matlis dual $D(M)$. By \cite[2.2(d)]{lyubeznik}, $M$ is a holonomic $\D$-module. However, by \cite[Theorem 2.4(c)]{lyubeznik}, $D(M)$, which has infinitely many associated primes, cannot even be a finitely generated $\D$-module, \emph{a fortiori} cannot be holonomic. Nevertheless, Theorem \ref{dRduals} implies that $D(M)$ has finite-dimensional de Rham cohomology.
\end{example}

\subsection{Graded duals over polynomial rings}\label{pregraded}

In this subsection, $R = k[x_1, \ldots, x_n]$ is the polynomial ring with its standard grading, {\it i.e.}, $\deg(x_i)=1$ for all $i$ and $\deg(c)=0$ for $c\in k$. By a \emph{graded} $R$-module we mean a $\mathbb{Z}$-graded module. 

An $R$-module homomorphism $f: M \rightarrow N$ between graded $R$-modules is \emph{graded} (or \emph{homogeneous}) if $f(M_n) \subseteq N_n$ for all $n \in \mathbb{Z}$; a submodule $N \subseteq M$ is a \emph{graded submodule} if there is a direct sum decomposition $N = \oplus_{l \in \mathbb{Z}} N_l$ as above such that the inclusion of $N$ in $M$ is a graded homomorphism. If $\{M_i\}$ is a collection of graded $R$-modules, their direct sum $\oplus_i M_i$ is also a graded $R$-module, with grading given by $(\oplus_i M_i)_l = \oplus_i (M_i)_l$ for all $l$. Graded $R$-modules together with graded homomorphisms form an Abelian category with enough projective and injective objects.

If $l \in \mathbb{Z}$ is fixed and $M$ is a graded $R$-module, the \emph{shifted} module $M(l)$ has the same underlying $R$-module as $M$ but a $\mathbb{Z}$-grading defined by $M(l)_n = M_{l+n}$ for all $n \in \mathbb{Z}$. If $M$ and $N$ are graded $R$-modules, we define $\grHom_R(M,N) = \oplus_{n \in \mathbb{Z}}\Hom_R(M,N)_n$ where $\Hom_R(M,N)_n$ is the Abelian group of graded $R$-module homomorphisms $M \rightarrow N(n)$ (such homomorphisms are called \emph{homogeneous of degree $n$}). Note that $\grHom_R(M,N)$ is a graded $R$-module; its underlying $R$-module is an $R$-submodule of $\Hom_R(M,N)$, and if $M$ is finitely generated as an $R$-module, we have the equality $\grHom_R(M,N) = \Hom_R(M,N)$.

If $I$ is a homogeneous ideal of $R$, then the local cohomology modules $H^j_I(R)$ are naturally graded, with the grading induced by the grading on $R$. In particular, $H^n_{\mathfrak{m}}(R)$ is naturally graded. More explicitly, each class $\left[\frac{1}{x^{i_1}_1\cdots x^{i_n}_n}\right]$ has degree $-(i_1 + \cdots + i_n)$ \cite[Example 13.5.3]{brodmann}.

\begin{convention}\label{threees}
We will consider $H^n_{(x_1,\dots,x_n)}(R)$ as the $R$-injective hull of $k$ and denote it by $\E$; when $R$ is a polynomial ring, $\E$ is endowed with the natural grading (in which $\deg(x_1^{-i_1}\cdots x_n^{-i_n})=-\sum_{j=1}^ni_j$). Throughout this paper, we will always consider this grading on $E$.
\end{convention}

\begin{remark}\label{starinjectivehull}
The $R$-module $E$, with the grading defined in Convention \ref{threees}, is isomorphic as an $R$-module (but not as a graded $R$-module) to the \emph{graded injective hull} ${}^*E$ of $k$ defined in \cite[\S 3.6]{bruns}. In fact, we have $E \cong {}^*E(n)$ as graded $R$-modules.
\end{remark}

Throughout this paper, we define the graded Matlis dual of a graded $R$-module as follows.
\begin{definition}\label{shifteddual}
Let $M$ be a graded $R$-module. The \emph{graded Matlis dual} of $M$ is the graded $R$-module $\DD(M) = \grHom_R(M, \E)$.
\end{definition}

As in the formal power series case, we have a $k$-linear \emph{residue map} $\sigma: \E \rightarrow k$, defined by projecting an element of $\E$ onto its $x_1^{-1} \cdots x_n^{-1}$-component. There is an analogue of Proposition-Definition \ref{sigmadual} that allows us to view elements of the dual $\DD(M)$ as maps to the field $k$:

\begin{prop}\cite[Proposition 3.6.16]{bruns}\label{gradedkdual}
Let $M$ be a graded $R$-module. There is an isomorphism of graded $R$-modules
\[
\Phi_M: \DD(M) \rightarrow \sideset{^*}{_k}{\Hom}(M(-n),k)
\]
defined by post-composition with the residue map $\sigma$ and functorial in $M$.
\end{prop}

A few remarks are in order.
\begin{remark}
\begin{enumerate}
\item Both forms of the graded Matlis dual will be useful for us, and so we will use the residue map, sometimes implicitly, to identify the two in what follows.
\item Our graded Matlis dual differs from the one in \cite[p. 141]{bruns} by a degree shift. The reason for this difference will become clear in Proposition \ref{euleriandual}. If one does not care whether the Eulerian property is preserved by the graded Matlis dual, one can use the non-shifted version throughout (only Proposition \ref{euleriandual} would become false).
\item The canonical evaluation map $\iota_M: M \rightarrow \DD(\DD(M))$ is an isomorphism of graded $R$-modules if and only if $M_l$ is a finite-dimensional $k$-space for all $l$.
\end{enumerate}
\end{remark}


\section{Graded $\D$-modules over polynomial rings}
\label{graded}
Throughout this section, $R = k[x_1, \ldots, x_n]$ is the polynomial ring with its standard grading. 

Hartshorne and Polini give an example \cite[Example 6.1]{hartpol} showing that Theorem \ref{dRdim and surjective map} fails in general in the polynomial case. Instead of holonomic $\D$-modules, we will restrict our attention to \emph{graded} $\D$-modules. We begin with the graded (polynomial) analogue of the Matlis duality theory for $\D$-modules recalled in subsection \ref{pregraded}.

\begin{definition}\label{gradeddmod}
Let $M$ be a (left) $\D$-module whose underlying $R$-module is given a grading $M = \oplus_{l \in \mathbb{Z}} M_l$. We say that $M$ is a \emph{graded $\D$-module} if for all $l \in \mathbb{Z}$ and $1 \leq i \leq n$, we have $\partial_i(M_l) \subseteq M_{l-1}$. There is an entirely analogous notion of graded \emph{right} $\D$-module.
\end{definition}

Chapters 1 and 2 of \cite{BookMethodsGradedRings} are a good reference for the general theory of (possibly non-commutative) graded rings and modules over them. The only non-commutative graded ring we will consider in this paper is $\D$.

\begin{example}\label{gradedds}
$R$ itself (with its standard grading) is a graded $\D$-module, as is $\E$. (Any degree shift of a graded $\D$-module is again a graded $\D$-module.) The graded $\D$-modules that are relevant for applications in \cite{hartpol} are local cohomology modules supported in homogeneous ideals (the previous examples are special cases of these). If $I \subseteq R$ is a homogeneous ideal, we know that $H^i_I(R)$ is a graded $R$-module (see \cite[Ch. 13]{brodmann} for a detailed discussion of the natural gradings on $H^i_I(R)$ and proofs that they all coincide) as well as a left $\D$-module, and the $\D$-module structure is compatible with the grading (it is easiest to see this if the \v{C}ech complex is used to compute $H^i_I(R)$).
\end{example}

\begin{example}\cite[Example 6.1]{hartpol}\label{badhart}
Let $R = k[x]$, let $d = \frac{d}{dx} \in \D$, and let $M$ be a free $R$-module $R \cdot e$ of rank $1$ generated by $e \in M$. We can give $M$ a structure of $\D$-module by setting $de = x^2e$ and extending by $R$-linearity to all of $M$. In \cite[Example 6.1]{hartpol}, it is proved that this $\D$-module is holonomic but fails to satisfy Theorem \ref{dRdim and surjective map}. We observe that $M$ is \emph{not} a graded $\D$-module. Indeed, the formula $de = x^2e$ shows that $d$ would be required to act simultaneously as an operator of degree $-1$ and an operator of degree $2$, which is absurd.
\end{example}

If $M$ is a graded $\D$-module, its Matlis dual $\DD(M)$ can be endowed with a (left) graded $\D$-module structure. We will do this in two equivalent ways, corresponding to the two sides of the isomorphism in Proposition \ref{gradedkdual} (both will be useful).

Ignoring the gradings for a moment, if $M$ and $N$ are any two left $\D$-modules, we can define a left $\D$-module structure on $\Hom_R(M,N)$ extending the natural $R$-module structure by setting
\begin{equation}\label{hottastructure}
(\partial_i \cdot \varphi)(m) = \partial_i \cdot \varphi(m) - \varphi(\partial_i \cdot m)
\end{equation}
for $i = 1 , \ldots, n$ and all $m \in M$ and $\varphi \in \Hom_R(M,N)$ \cite[Proposition 1.2.9]{hotta}. Since $\D$ is generated over $R$ by the derivations $\partial_i$, this formula gives a well-defined left $\D$-module structure (simply extend by $\D$-linearity) as long as the relations among elements of $R$ and the $\partial_i$ are preserved. (See \cite[Lemma 1.2.1]{hotta} for a precise statement of what this means.) This $\D$-module structure on $\Hom$ is well-known and originates in Rinehart's thesis \cite{rinehart}. If $M$ and $N$ are \emph{graded} $\D$-modules, it is clear from \eqref{hottastructure} that the left $\D$-structure on $\Hom_R(M,N)$ induces a left $\D$-structure on $\grHom_R(M, N)$. Taking $N = \E$, we see that whenever $M$ is a graded $\D$-module, so also is $\DD(M)$.

On the other hand, we can define a graded $\D$-module structure directly on $\sideset{^*}{_k}{\Hom}(M(-n),k)$: since each differential operator in $\D$ acts on $M$ via a $k$-linear map, we can decree that such differential operators act on $\sideset{^*}{_k}{\Hom}(M(-n),k)$ by pre-composition. This construction is more explicitly a ``dual'' of the original $\D$-module structure on $M$. However, it is naturally a \emph{right} $\D$-module structure, so in order to compare the two structures, we will need to use the following \emph{transposition} operation:

\begin{definition}\label{transpose}
\begin{enumerate}[(a)]
\item The \emph{standard transposition} $\tau: \D \rightarrow \D$ is defined by
\[
\tau(f \partial_1^{i_1} \cdots \partial_n^{i_n}) = (-1)^{i_1 + \cdots + i_n}\partial_1^{i_1} \cdots \partial_n^{i_n}f
\]
for all $f \in R$, extended to all of $\D$ by $k$-linearity (observe that the same operation makes sense for formal power series)
\item Let $M$ be a right $\D$-module. The \emph{transpose} $M^{\tau}$ of $M$ is the left $\D$-module defined as follows: we have $M^{\tau} = M$ as Abelian groups, and the left $\D$-action $\ast$ on $M^{\tau}$ is given by $\delta \ast m = m \cdot \tau(\delta)$ for all $\delta \in \D$ and $m \in M (=M^{\tau})$.
\end{enumerate}
\end{definition}

\begin{remark}
\begin{enumerate}
\item If $M$ is a left $\D$-module, a completely analogous transposition operation produces a right $\D$-module.
\item $\tau^2 = \mathrm{id}_{\D}$; hence applying this operation twice recovers the original (right or left) $\D$-module. 
\item If $M$ is a graded right (resp. left) $\D$-module, its transpose $M^{\tau}$ is clearly a graded left (resp. right) $\D$-module.
\item We have $\tau(\delta_1\delta_2)=\tau(\delta_2)\tau(\delta_1)$ for all $\delta_1,\delta_2\in \D$. 
\end{enumerate}
\end{remark}

Given any graded $\D$-module $M$, a left $\D$-module structure on $\grHom_k(M(-n),k)$ extending the natural $R$-module structure can be defined by setting
\begin{equation}\label{transposestructure}
(\partial_i \cdot \lambda)(m)= \lambda(\tau(\delta) \cdot m)
\end{equation}
for $i = 1, \ldots, n$ and all $m \in M(-n)$ and $\lambda \in \grHom_k(M(-n),k)$. It is not hard to check that the resulting $\D$-module structure is well-defined and \emph{graded} by direct calculation. However, this also results by ``transport of structure'' from the following proposition, since by Proposition \ref{gradedkdual}, $\Phi_M: \grHom_R(M, \E) \rightarrow \grHom_k(M(-n),k)$ is an isomorphism of graded $R$-modules.

\begin{prop}\label{gradedDdual}
Let $M$ be a graded $\D$-module. For all $i$, there is a commutative diagram
\[
\begin{CD}
\grHom_R(M, \E)     @> \Phi_M >>  \grHom_k(M(-n),k)\\
@VV \partial_i V        @VV\partial_i V\\
\grHom_R(M, \E)     @> \Phi_M >>  \grHom_k(M(-n),k)
\end{CD}
\]
where the left vertical arrow is given by \eqref{hottastructure} and the right vertical arrow is given by \eqref{transposestructure}.
\end{prop}

The upshot of Proposition \ref{gradedDdual} is that if we identify $\grHom_R(M, \E)$ with $\grHom_k(M(-n),k)$ using the residue map (Proposition \ref{gradedkdual}), it does not matter whether we use \eqref{hottastructure} or \eqref{transposestructure} to make $\DD(M)$ into a \emph{graded} (left) $\D$-module. Both viewpoints will be useful to us below and we will freely switch between them. In either case we refer to $\DD(M)$ as the \emph{graded $\D$-module Matlis dual} of $M$.

\begin{proof}
Let $\varphi \in \grHom_R(M, \E)$ and $m \in M$ be given. By \eqref{hottastructure}, we have $(\partial_i \cdot \varphi)(m) = \partial_i \cdot \varphi(m) - \varphi(\partial_i \cdot m)$. Applying $\Phi_M$, which is post-composition with the residue map $\sigma$, we see that
\[
\Phi_M(\partial_i \cdot \varphi)(m) = \sigma(\partial_i \cdot \varphi(m) - \varphi(\partial_i \cdot m)) = \sigma(\partial_i \cdot \varphi(m)) - \sigma(\varphi(\partial_i \cdot m)).
\]
However, since $\varphi(m) \in E$, $\partial_i \cdot \varphi(m)$ cannot have a nonzero $x_1^{-1} \cdots x_n^{-1}$-component (after partial differentiation, the variable $x_i$ must have degree $-2$ or lower). Therefore $\sigma(\partial_i \cdot \varphi(m)) = 0$ and $\Phi_M(\partial_i \cdot \varphi)(m) = -\sigma(\varphi(\partial_i \cdot m))$, which is exactly $(\partial_i \cdot \Phi_M(\varphi))(m)$ (the minus sign arises from the application of the transpose $\tau$).
\end{proof}

\begin{example}\label{onevar}
Let $\nu: E \rightarrow \grHom_R(R, E)$ be the canonical isomorphism of graded $R$-modules defined by $\eta \mapsto (1 \mapsto \eta)$. We use \eqref{hottastructure} to calculate $\partial_i \cdot \nu(\eta)$ for $i = 1, \ldots, n$ and $\eta \in E$: 
\[
(\partial_i \cdot \nu(\eta))(r) = \partial_i \cdot \nu(\eta)(r) - \nu(\eta)(\partial_i \cdot r) = \partial_i \cdot (r \eta) - (\partial_i \cdot r)\eta = r \cdot \partial_i \eta = \nu(\partial_i \eta)(r),
\]
from which it follows that $\nu$ is $\D$-linear. Therefore the graded $\D$-module Matlis dual $\DD(R)$ of $R$ is just $E$ with its usual left $\D$-module structure. 
\end{example}


As we will see, the operation $\DD$ enjoys some desirable properties. For instance, it preserves Eulerianness, whose definition we recall below. 

\begin{definition}[Definition 2.1 and Proposition 3.1 in \cite{MaZhangEulerianGradedDModules}]
A graded $\D$-module $M$ is called {\it Eulerian} if for each homogeneous element $z\in M$ we have
\[(\sum_{i=1}^nx_i\partial_i)z=\deg(z)z.\]
\end{definition}

\begin{prop}\label{euleriandual}
If $M$ is an Eulerian graded $\D$-module, then so is $\DD(M)$.
\end{prop}
\begin{proof}
For each $\lambda\in \DD(M)_{l}=\Hom_k(M_{-l-n},k)$ and each $z\in M_{-l-n}$, we have
\begin{align}
(\sum_{i=1}^nx_i\partial_i\cdot \lambda)(z)&=\lambda(\tau(\sum_{i=1}^nx_i\partial_i)z) \notag\\
&=\lambda((-\sum_{i=1}^n\partial_ix_i)z) \notag\\
&=-\lambda(\sum_{i=1}^n(x_i\partial_i+1)z)\notag\\
&=-\lambda((\sum_{i=1}^nx_i\partial_i)z+nz)\notag\\
&=-\lambda((-l-n)z+nz)\ {\rm (since\ }M\ {\rm is\ Eulerian)}\notag\\
&=l \lambda(z)\notag\\
&=\deg(\lambda)\lambda(z)\notag
\end{align}
Therefore, $\sum_{i=1}^nx_i\partial_i\cdot \lambda=\deg(\lambda)\lambda$ and hence $\DD(M)$ is Eulerian.
\end{proof}

We now turn to the question of which $k$-linear maps between graded $\D$-modules can be dualized. If $\delta: M \rightarrow N$ is homogeneous of any degree (if there exists $d$ such that $\delta(M_l) \subseteq N_{l+d}$ for all $l$), then whenever $\lambda \in \grHom_k(N(-n),k)$, the composite $\delta \circ \lambda$ belongs to $\grHom_k(M(-n),k)$. More generally, this is true whenever $\delta \in \grHom_k(M,N)$ (that is, $\delta$ is a finite sum of $k$-linear maps, each homogeneous of a fixed degree), inspiring the following:

\begin{definition}\label{dualmap}
Let $M$ and $N$ be graded $\D$-modules, and suppose that $\delta \in \grHom_k(M,N)$. We define the \emph{Matlis dual} $\delta^* \in \grHom_k(\DD(N), \DD(M))$ of $\delta$ by pre-composition with $\delta$: that is, $\delta^*(\lambda) = \lambda \circ \delta$ for all $\lambda \in \DD(N) = \grHom_k(N(-n),k)$.
\end{definition}

We remark that since the definition of $\delta^*$ is simply pre-composition, if $\delta$ is also $R$-linear (that is, $\delta \in \grHom_R(M,N)$), then $\delta^*$ is again $R$-linear; moreover, if $\delta$ is $\D$-linear, $\delta^*$ is again $\D$-linear (the proof is the same as that of Lemma \ref{dualofdlinear}). In particular, the graded Matlis dual operation is a contravariant functor from the category of graded (left) $\D$-modules to itself. 

If $M$ is a graded $\D$-module, we can discuss the de Rham cohomology spaces of $M$ and its graded Matlis dual $\DD(M)$, and in particular, we can ask whether the analogue of Theorem \ref{dRduals} is true for a graded holonomic $\D$-module $M$. In fact, a more general statement is true: such an analogue holds for any graded $\D$-module $M$ whose de Rham cohomology spaces are finite-dimensional. (In the formal power series case, the holonomicity of $M$ is used in an essential way.)  

\begin{thm}\label{dRdualsg}
Let $M$ be a graded $\D$-module. For all $i$ such that $H^i_{dR}(M)$ is a finite-dimensional $k$-space, we have
\[
H^i_{dR}(M)^{\vee} \cong H^{n-i}_{dR}(\DD(M)) 
\]
as $k$-spaces.
\end{thm}

\begin{proof}
We write the de Rham complex $M \otimes \Omega_R^{\bullet}$ as
\[
0 \rightarrow M^0 \xrightarrow{\delta^0} M^1 \xrightarrow{\delta^1} \cdots \xrightarrow{\delta^{n-1}} M^n \rightarrow 0,
\]
where $M^i$ is a direct sum of $n \choose i$ copies of $M$ for all $i$. Observe that each $M^i$ is a graded $\D$-module and each $\delta^i$ belongs to $\grHom_k(M^i, M^{i+1})$ (in fact, $\delta^i$ is homogeneous of degree $-1$). In the category of complexes of $k$-spaces, this complex decomposes as a direct sum
\[
\oplus_{l \in \mathbb{Z}}(0 \rightarrow M^0_l \xrightarrow{\delta^0_l} M^1_{l-1} \xrightarrow{\delta^1_{l-1}} \cdots \xrightarrow{\delta^{n-1}_{l-n+1}} M^n_{l-n} \rightarrow 0)
\]
where $\delta^i_j$ denotes the restriction of $\delta^i$ to the degree $j$ component of $M^i$. (Write $M^{\bullet}_l$ for the $l$th summand, so that $M \otimes \Omega_R^{\bullet} = \oplus_{l \in \mathbb{Z}} M^{\bullet}_l$ as complexes.) We can take the graded Matlis dual of this entire complex, obtaining the complex
\[
\DD(M \otimes \Omega_R^{\bullet}) = (0 \rightarrow \DD(M^n) \xrightarrow{(\delta^{n-1})^*} \DD(M^{n-1}) \xrightarrow{(\delta^{n-2})^*} \cdots \xrightarrow{(\delta^0)^*} \DD(M^0) \rightarrow 0),
\]
again a complex whose objects are graded $\D$-modules and whose differentials are $k$-linear and homogeneous of degree $-1$, but now with \emph{homological} indexing. For all $i$, we have $\DD(M^i)_{l} = (M^i_{-l-n})^{\vee}$, and the complex $\DD(M \otimes \Omega_R^{\bullet})$ decomposes (in the category of complexes of $k$-spaces) as a direct sum
\[
\oplus_{l \in \mathbb{Z}}(0 \rightarrow (M^n_{-l-n})^{\vee} \xrightarrow{(\delta^{n-1}_{-l-n+1})^{\vee}} (M^{n-1}_{-l-n+1})^{\vee} \xrightarrow{(\delta^{n-2}_{-l-n+2})^{\vee}} \cdots \xrightarrow{(\delta^0_{-l})^{\vee}} (M^0_{-l})^{\vee} \rightarrow 0),
\]
that is, $\DD(M \otimes \Omega_R^{\bullet}) = \oplus_{l \in \mathbb{Z}} (M^{\bullet}_{-l})^{\vee}$, which is just $\oplus_{l \in \mathbb{Z}} (M^{\bullet}_{l})^{\vee}$ as a complex of $k$-spaces with the gradings forgotten. In the category of $k$-spaces, the (contravariant) $k$-dual functor is exact, and taking the (co)homology objects of a complex commutes with arbitrary direct sums. It follows that
\begin{align*}
H_{dR}^i(M)^{\vee} &= (h^i(M \otimes \Omega^{\bullet}_R))^{\vee} \\
&= (h^i(\oplus_{l \in \mathbb{Z}} M^{\bullet}_{l}))^{\vee}\\
&\cong (\oplus_{l \in \mathbb{Z}} h^i(M_{l}^{\bullet}))^{\vee} \\
&\cong \oplus_{l \in \mathbb{Z}} (h^i(M_{l}^{\bullet}))^{\vee} \\
&\cong \oplus_{l \in \mathbb{Z}} h_i((M_{l}^{\bullet})^{\vee}) \\
&\cong h_i(\oplus_{l \in \mathbb{Z}} (M_{l}^{\bullet})^{\vee}) \\
&= h_i(\DD(M \otimes \Omega_R^{\bullet}))
\end{align*}
as $k$-spaces. The isomorphism $(\oplus_{l \in \mathbb{Z}} h^i(M_{l}^{\bullet}))^{\vee} \cong \oplus_{l \in \mathbb{Z}} (h^i(M_{l}^{\bullet}))^{\vee}$ holds due to our assumption that $H_{dR}^i(M)$ is a finite-dimensional $k$-space (we always have $(\oplus_{l \in \mathbb{Z}} h^i(M_{l}^{\bullet}))^{\vee} \cong \prod_{l \in \mathbb{Z}} (h^i(M_{l}^{\bullet}))^{\vee}$, but since both sides are finite-dimensional, the direct product on the right-hand side coincides with the direct sum). 

It now suffices to show that
\[
h_i(\DD(M \otimes \Omega_R^{\bullet})) \cong h^{n-i}(\DD(M) \otimes \Omega_R^{\bullet}) \, \, (= H^{n-i}_{dR}(\DD(M)))
\]
as $k$-spaces, for all $i$. We first compute the differentials in the complex $\DD(M \otimes \Omega_R^{\bullet})$.  Let $i$ be given, and consider the differential $d^i: M \otimes \Omega_R^i \rightarrow M \otimes \Omega_R^{i+1}$.  An element of $M \otimes \Omega_R^i$ is a sum of terms of the form $m_{j_1 \cdots j_i} \, dx_{j_1} \wedge \cdots \wedge dx_{j_i}$ where $1 \leq j_1 < \cdots < j_i \leq n$, and the formula for $d^i$ is 
\[
d^i(m\, dx_{j_1} \wedge \cdots \wedge dx_{j_i}) = \sum_{s=1}^n \partial_s(m)\, dx_s \wedge dx_{j_1} \wedge \cdots \wedge dx_{j_i}.
\]
Now consider the graded Matlis dual of this differential.  Since the graded Matlis dual commutes with finite direct sums, we can identify $\DD(M \otimes \Omega_R^i)$ with a direct sum of $n \choose i$ copies of $\DD(M)$, again indexed by the $dx_{j_1} \wedge \cdots \wedge dx_{j_i}$.  If $\varphi \in \DD(M)$, we have the formula
\[
(d^i)^*(\varphi \, dx_{j_1} \wedge \cdots \wedge dx_{j_{i+1}}) = \sum_{s=1}^{i+1} (-1)^{s-1} \partial_{j_s}^*(\varphi)\, dx_{j_1} \wedge \cdots \wedge \widehat{dx_{j_s}} \wedge \cdots \wedge dx_{j_{i+1}}.
\]
Recall that when $\DD(M)$ is viewed as a left $\D$-module via transposition, $\partial_i$ acts on $\DD(M)$ via the map $-\partial_i^*$. Therefore, it follows from the formula above that the complex $\DD(M \otimes \Omega_R^{\bullet})$ is the \emph{homological} Koszul complex $K_{\bullet}(\DD(M), \mathbf{-\partial})$ of $\DD(M)$ with respect to $-\partial_1, \ldots, -\partial_n$, and if we replace $-\partial_i$ with $\partial_i$ for all $i$, the homology objects are not affected. On the other hand, the de Rham complex $\DD(M) \otimes \Omega_R^{\bullet}$ is the \emph{cohomological} Koszul complex $K^{\bullet}(\DD(M); \mathbf{\partial})$ of $\DD(M)$ with respect to $\partial_1, \ldots, \partial_n$, and it is well-known \cite[Exercise 4.5.2]{weibel} that $h_i(K_{\bullet}(\DD(M), \mathbf{\partial})) \cong h^{n-i}(K^{\bullet}(\DD(M); \mathbf{\partial}))$ as modules over the commutative subring $k[\partial_1, \ldots, \partial_n] \subseteq \D$ (in particular, as $k$-spaces) for all $i$, completing the proof.
\end{proof}

\begin{example}\label{infdimgraded}
In general, even if a graded $\D$-module $M$ has finite-dimensional de Rham cohomology, its graded pieces $M_l$ may be infinite-dimensional as $k$-spaces and thus fail to be isomorphic to their duals or double duals, and so the isomorphisms in the proof of Theorem \ref{dRdualsg} hold only at the level of cohomology. For example, let $R = k[x_1, \ldots, x_n]$ with $n \geq 2$ and let $M$ be the local cohomology module $H^1_{(x_1)}(R)$. Since $R$ has its standard grading and $(x_1) \subseteq R$ is a homogeneous ideal, $M$ is a graded $\D$-module. Concretely, $M$ takes the form
\[
k[x_2, \ldots, x_n][x_1^{-1}] \cong \oplus_{l \geq 1} k[x_2, \ldots, x_n] \cdot \frac{1}{x_1^l}
\]
(we see this by computing $H^1_{(x_1)}(R)$ using the \v{C}ech complex) where the $R$-module structure is defined by setting $x_1^i = 0$ for $i \geq 0$. Each term $\frac{x_2^{i_2}\cdots x_n^{i_n}}{x_1^j}$ has degree $i_2 + \cdots + i_n - j$, and for each $l$, there are infinitely many tuples $(i_2, \ldots, i_n, j)$ such that $i_2 + \cdots + i_n - j = l$. Therefore, each component of $M$ is an infinite-dimensional $k$-space.
\end{example}

It follows from Theorem \ref{dRdualsg} that if $M$ is a graded $\D$-module with finite-dimensional de Rham cohomology (for example, a graded holonomic $\D$-module), the graded Matlis dual $\DD(M)$ has finite-dimensional de Rham cohomology. An important property of holonomic $\D$-modules is that, by Theorem \ref{dRfindim}, they have finite dimensional de Rham cohomology. It is natural to ask whether the graded Matlis dual of a graded holonomic $\D$-module is also holonomic. This turns out {\it not} to be the case, as shown in Example \ref{example: holonomicity not preserved}. We note that since $\DD(M)$ always has finite-dimensional de Rham cohomology, Theorem \ref{gradedhart} below applies to it, even in cases where it is not holonomic.

Before proceeding to Example \ref{example: holonomicity not preserved}, we need a result due to Hellus and St\"uckrad \cite{HellusStuckrad}.
\begin{remark}
\label{remark: hellus-stuckrad}
The proof of \cite[Lemma 1.1]{HellusStuckrad} shows that, given any commutative ring $A$ and any $a_1,a_2\in A$, there are elements $\{b_i\in A\mid i\in \NN\}$ such that $b_i$ is either 1 or 0 for each $i$ and the map defined via
\[\left[\frac{1}{a_1^i}\right]\mapsto \sum_{j=1}^{i}\left[\frac{b_j}{a_1^{i-j+1}a_2^j}\right]\]
induces a surjective $A$-module homomorphism $\varphi: H^1_{(a_1)}(A)\to H^2_{(a_1,a_2)}(A)$ (where brackets denote classes in local cohomology viewed as cohomology of the \v{C}ech complex). When $A$ is graded and $a_1,a_2$ are homogeneous, it is clear that $\varphi\in \sideset{^*}{_A}{\Hom}(H^1_{(a_1)}(A), H^2_{(a_1,a_2)}(A))$. By induction, given homogeneous elements $a_1,\dots,a_m\in A$, there exists a surjective $A$-module homomorphism $\varphi\in \sideset{^*}{_A}{\Hom}(H^1_{(a_1)}(A),H^m_{(a_1,\dots,a_m)}(A))$.
\end{remark} 
 
\begin{example}
\label{example: holonomicity not preserved}
Let $R=k[x_1,\dots,x_n]$ with $n\geq 2$. Let $M=H^1_{(x_1)}(R)$ and let $\mathbb{U}_+(x_1)$ denote the set of homogeneous prime ideals of $R$ that do {\it not} contain $x_1$. Then we claim that 
\begin{equation}
\label{infinitely ass primes}
\mathbb{U}_+(x_1) \subseteq \Ass_R(\DD(M)).
\end{equation} 
Since $\mathbb{U}_+(x_1)$ contains infinitely many elements, \eqref{infinitely ass primes} will imply that $\DD(M)$ is {\it not} holonomic since, by \cite[Theorem 2.4(c)]{lyubeznik}, a holonomic $\D$-module has only finitely many associated primes. (Note the similarity of this example to Example \ref{hellus}.)

\begin{proof}[Proof of \eqref{infinitely ass primes}]
Let $\fp$ be a homogeneous prime ideal of $R$ that does not contain $x_1$. Then $x_1$ is part of a homogeneous system of parameters for $R/\fp$; in fact, we can choose homogeneous elements $x_1,y_2,\dots, y_d\in R$ whose images in $R/\fp$ form a homogeneous system of parameters (where $d=\dim(R/\fp)$). By Remark \ref{remark: hellus-stuckrad}, there is a surjective map $\varphi\in \grHom_R(H^1_{(x_1)}(R),H^d_{(x_1,y_2,\dots,y_d)}(R))$ and hence an injective map 
\[\varphi^*: \DD(H^d_{(x_1,y_2,\dots,y_d)}(R))\hookrightarrow \DD(H^1_{(x_1)}(R)).\]
Thus, to show that $\fp\in \Ass_R(\DD(H^1_{(x_1)}(R)))$, it suffices to show that there is an injection $R/\fp \hookrightarrow \DD(H^d_{(x_1,y_2,\dots,y_d)}(R))$. To this end, it is enough to prove that 
\begin{enumerate}
\item $\grHom_R(R/\fp,\DD(H^d_{(x_1,y_2,\dots,y_d)}(R)))\neq 0$, and that
\item $\Hom_R(R/\fq, \DD(H^d_{(x_1,y_2,\dots,y_d)}(R)))=0$ for any $\fq$ properly containing $\fp$. 
\end{enumerate}
By the graded version of adjunction of $\Hom$ and $\otimes$ \cite[Proposition 2.4.9]{BookMethodsGradedRings}, we have 
\begin{align}
\grHom_R(R/\fp,\DD(H^d_{(x_1,y_2,\dots,y_d)}(R))) &= \grHom_R(R/\fp, \grHom_R(H^d_{(x_1,y_2,\dots,y_d)}(R), \E))\notag\\
&\cong \grHom_R(R/\fp\otimes_R H^d_{(x_1,y_2,\dots,y_d)}(R), \E)\notag\\
&\cong \grHom_R(H^d_{(x_1,y_2,\dots,y_d)}(R/\fp), \E)\notag\\
&\neq 0,\notag
\end{align}
since $H^d_{(x_1,y_2,\dots,y_d)}(R/\fp) \neq 0$, proving the first statement. (The isomorphism $R/\fp\otimes_R H^d_{(x_1,y_2,\dots,y_d)}(R) \cong H^d_{(x_1,y_2,\dots,y_d)}(R/\fp)$ holds because the top local cohomology functor $H^d_{(x_1,y_2,\dots,y_d)}$ is right-exact.)

We also have $\grHom_R(H^d_{(x_1,y_2,\dots,y_d)}(R), \E)\subseteq \Hom_R(H^d_{(x_1,y_2,\dots,y_d)}(R), \E)$ and so, if $\fp \subsetneq \fq$, 
\begin{align}
\Hom_R(R/\fq, \DD(H^d_{(x_1,y_2,\dots,y_d)}(R))) &\subseteq \Hom_R(R/\fq, \Hom_R(H^d_{(x_1,y_2,\dots,y_d)}(R), \E))\notag\\
& \cong \Hom_R(R/\fq\otimes_R H^d_{(x_1,y_2,\dots,y_d)}(R), \E)\notag\\
&\cong \Hom_R(H^d_{(x_1,y_2,\dots,y_d)}(R/\fq), \E)\notag\\
&=0,\notag
\end{align}
proving the second statement ($H^d_{(x_1,y_2,\dots,y_d)}(R/\fq)=0$ since $\dim(R/\fq)<d$). 
\end{proof}
\end{example}

\begin{lem}\label{gmatlisref}
Let $M$ and $N$ be graded $R$-modules, and suppose that $N_l$ is a finite-dimensional $k$-space for all $l$. Suppose furthermore that $\varphi \in \grHom_R(N, \DD(M))$ is injective. Then the composite
\[
M \xrightarrow{\iota_M} \DD(\DD(M)) \xrightarrow{\varphi^*} \DD(N),
\]
where $\iota_M$ is the canonical evaluation map, is surjective.
\end{lem}

Of course $\varphi^*$ is always surjective; the claim is that, in this case, it remains surjective when restricted to the image of $M$.

\begin{proof}
Consider the Matlis dual of the displayed composite, which factors as
\[
\DD(\DD(N)) \xrightarrow{\varphi^{**}} \DD(\DD(\DD(M))) \xrightarrow{\iota_M^*} \DD(M),
\]
and pre-compose it with the evaluation map $\iota_N$, which is an isomorphism by assumption. It suffices to show that the resulting composite
\[
N \xrightarrow{\iota_N} \DD(\DD(N)) \xrightarrow{\varphi^{**}} \DD(\DD(\DD(M))) \xrightarrow{\iota_M^*} \DD(M)
\]
coincides with $\varphi$ (and is therefore injective), since if $\varphi^{*} \circ \iota_M$ were not surjective, its dual could not be injective. 

Let $n \in N$ be given. The element $\iota_N(n) \in \DD(\DD(N))$ is the map $\DD(N) \rightarrow \E$ defined by evaluation at $n$, and therefore the element $\varphi^{**}(\iota_N(n)) \in \DD(\DD(\DD(M)))$ is the map $\DD(\DD(M)) \rightarrow \E$ defined by evaluation at $\varphi(n) \in \DD(M)$. But then $\iota_M^*(\varphi^{**}(\iota_N(n)))$ is simply $\varphi(n)$, since $\iota_M^*$ is the dual of the evaluation map. It follows that the composite $\iota_M^* \circ \varphi^{**} \circ \iota_N$ coincides with $\varphi$, as claimed.
\end{proof}

Finally, we prove our main result, the graded analogue of Theorem \ref{dRdim and surjective map}. Note that we do not need any assumption on holonomicity, because in the graded case, Theorem \ref{dRdualsg} is valid under a weaker hypothesis.

\begin{thm}\label{gradedhart}
Let $M$ be a graded $\D$-module such that $H^n_{dR}(M)$ is a finite-dimensional $k$-space. Then
\[\dim_k(H^n_{dR}(M))=\max\{s\in\mathbb{N}\mid \exists\ {\rm a\ surjective\ }\varphi\in \Hom_{\D}(M,E^s) \}.\]
\end{thm}

Note that we do \emph{not} claim that the map $M \rightarrow E^s$ is a homomorphism of \emph{graded} $\D$-modules. (If we remember the grading on $E$, the map will be an element of $\grHom_R(M, \E^s)$.)

\begin{proof}
Let $t = \dim_k H^n_{dR}(M)$. By Theorem \ref{dRdualsg}, we have $t = \dim_k H^0_{dR}(\DD(M))$. By Lemma \ref{drzero}, there exists an injective $\D$-module homomorphism $i: R^t \rightarrow \DD(M)$. 

Recall from the proof of Lemma \ref{drzero} that $i$ is constructed by choosing a basis $\{\mu_1, \ldots, \mu_t\}$ for $H^0_{dR}(\DD(M))$ and defining $i(r_1, \ldots, r_t) = r_1\mu_1 + \cdots + r_t\mu_t$. It is clear that $i$ is not, in general, a graded homomorphism, but we can show that $i \in \grHom_R(R^t, \DD(M))$, as follows. Each $\partial_i$ can be viewed as a graded homomorphism of $k$-spaces $\DD(M) \rightarrow \DD(M)(-1)$, and so its kernel is a graded $k$-subspace of $\DD(M)$. Therefore $H^0_{dR}(\DD(M)) = \cap_{i=1}^n \ker(\partial_i)$ is also a graded $k$-subspace of $\DD(M)$, from which it follows that every homogeneous component of $\mu_i$, for all $i$, belongs again to $H^0_{dR}(\DD(M))$. By decomposing each $\mu_i$ into its homogeneous components, we can write $i$ as a finite sum of the maps $(r_1, \ldots, r_t) \mapsto r_j\mu_{j,l}$ (where $\mu_{j,l}$ is the degree $l$ component of $\mu_j$), each of which is $k$- (indeed, $\D$-) linear and homogeneous of some degree.

Now we can take the Matlis dual of $i$, obtaining a map $i^*: \DD(\DD(M)) \rightarrow \DD(R^t) = \E^t$. Since $i$ is $\D$-linear, so also is $i^*$, by the same argument given in the proof of Lemma \ref{dualofdlinear}. The graded components of $R^t$ are finite-dimensional $k$-spaces, so we can apply Lemma \ref{gmatlisref}, obtaining a surjection $M \xrightarrow{\iota_M} \DD(\DD(M)) \xrightarrow{i^*} \E^t$ (in fact, $i^* \circ \iota_M \in \grHom_R(M, \E^t)$). The evaluation map $\iota_M$ is $\D$-linear (the proof is the same as in Lemma \ref{evaldlinear}), so $i^* \circ \iota_M$ is $\D$-linear as well. It follows that $t \leq s$.

To prove the converse inequality, consider a surjective $\D$-linear homomorphism $M \rightarrow E^s$. If $K$ denotes the kernel of this homomorphism, we have a short exact sequence of $\D$-modules $0 \rightarrow K \rightarrow M \rightarrow E^s \rightarrow 0$. The corresponding long exact sequence of de Rham cohomology terminates with a surjection 
\[
H^n_{dR}(M) \rightarrow H^n_{dR}(E^s)
\]
of $k$-spaces. By Example \ref{poincare}, $\dim_k H^n_{dR}(E^s) = s$, from which it follows that $t = \dim_k H^n_{dR}(M) \geq s$, completing the proof.
\end{proof}

\section{An alternate proof of Theorem \ref{dRdim and surjective map}}\label{hartpolalt}
\label{section: HP theorem}
In this section, we give an alternate proof of Theorem \ref{dRdim and surjective map}. In \cite{hartpol} it is stated that this theorem is dual in a sense to Lemma \ref{drzero}; our proof makes that duality explicit. Throughout this section, $R$ denotes the formal power series ring $k[[x_1, \ldots, x_n]]$, and $D$ the Matlis dual functor. 

We will need a local analogue of Lemma \ref{gmatlisref}.

\begin{lem}\label{matlisref}
Let $M$ be an $R$-module, and let $N$ be a Matlis reflexive $R$-module. Suppose that $\varphi: N \rightarrow D(M)$ is an injective $R$-module map. Then the composite
\[
M \xrightarrow{\iota_M} D(D(M)) \xrightarrow{D(\varphi)} D(N),
\]
where $\iota_M$ is the canonical evaluation map, is surjective.
\end{lem}

\begin{proof}
The evaluation map $\iota_N: N \rightarrow D(D(N))$ is an isomorphism by the assumption on $N$, and the functor $D$ is exact. Therefore the proof of Lemma \ref{gmatlisref} also works in this case.
\end{proof}

\begin{proof}[An alternate proof of Theorem \ref{dRdim and surjective map}]
Let $M$ be a holonomic $\D$-module and let $t = \dim_k H^n_{dR}(M)$. By Theorem \ref{dRduals}, $t = \dim_k H^0_{dR}(D(M))$. By Lemma \ref{drzero}, there exists an injective $\D$-module homomorphism $R^t \rightarrow D(M)$. Since $R^t$ is a finitely generated (and hence Matlis reflexive) $R$-module, by Lemma \ref{matlisref}, the composite $M \xrightarrow{\iota_M} D(D(M)) \rightarrow D(R^t) = E^t$ is a surjective $R$-module homomorphism. By Lemma \ref{evaldlinear}, the evaluation map $M \xrightarrow{\iota_M} D(D(M))$ is $\D$-linear, and by Lemma \ref{dualofdlinear}, the map $D(D(M)) \rightarrow D(R^t)$ is $\D$-linear as well, so the composite is $\D$-linear. It follows that $t \leq s$. The proof that $t \geq s$ is identical to the argument given in the proof of Theorem \ref{gradedhart}, since Example \ref{poincare} applies to the formal power series case as well as the polynomial case.
\end{proof}

\section{The de Rham cohomology of a graded Matlis dual}\label{exttor}

Let $R = k[[x_1, \ldots, x_n]]$ and $\D = \D(R,k)$. As observed by Hartshorne and Polini \cite[Corollary 5.2]{hartpol}, if $M$ is a holonomic $\D$-module, the maximal integer $s$ such that there exists a surjective $\D$-linear map $M \rightarrow E^s$ is the dimension of the $k$-space $\Hom_{\D}(M, E)$. Therefore, Theorem \ref{dRdim and surjective map} asserts that $\dim_k(H^n_{dR}(M)) = \dim_k(\Hom_{\D}(M,E))$. It is natural to ask whether there is a similar connection between $H^{n-i}_{dR}(M)$ and $\Ext^i_{\D}(M,E)$ for $i > 0$, and a result in the affirmative was proved by Lyubeznik.

\begin{thm}\label{gennady}
Let $R = k[[x_1, \ldots, x_n]]$ and $\D = \D(R,k)$, and let $M$ be a $\D$-module.
\begin{enumerate}[(a)]
\item For all $i \geq 0$, $H^i_{dR}(D(M)) \cong \Ext^i_{\D}(M, E)$ as $k$-spaces. \cite[Corollary 4.1, Theorem 4.2]{gennadyext}
\item If $M$ is holonomic, $\dim_k(H^{n-i}_{dR}(M)) = \dim_k(\Ext^i_{\D}(M,E))$ for all $i \geq 0$. \cite[Theorem 1.3]{gennadyext}
\end{enumerate}
\end{thm}

Of course, part (b) of Theorem \ref{gennady} follows immediately from part (a) and Theorem \ref{dRduals}. The proof of part (a) uses the following well-known fact (proved using an explicit free resolution of $R$ as a $\D$-module) that we will also need.

\begin{prop}\label{de Rham as ext}
Let $R$ be either $k[x_1, \ldots, x_n]$ or $k[[x_1, \ldots, x_n]]$, let $\D = \D(R,k)$, and let $M$ be a $\D$-module. Then $H^i_{dR}(M) \cong \Ext^i_{\D}(R, M)$ as $k$-spaces for all $i \geq 0$.
\end{prop}

Since Proposition \ref{de Rham as ext} is also true for $\mathscr{D}_X$-modules over a complex-analytic manifold $X$, Theorems \ref{gennady} and \ref{gradedgennady} are local algebraic versions of classical duality results in $\D$-module theory. See, for example, Kashiwara's \cite[Proposition 5.1]{kashiwara} for the complex-analytic version of these statements.

The following Theorem \ref{gradedgennady} is an analogue of Theorem \ref{gennady} for graded $\D$-modules over polynomial rings. We remark that our proof of Theorem \ref{gradedgennady} (in the graded setting) adapts easily to give an alternate proof of Theorem \ref{gennady} as well.

Since the graded Matlis dual is defined in terms of the functor $\grHom$, the correct statement will involve its right derived functors $\grExt$ \cite[p. 28]{BookMethodsGradedRings}; a true statement involving only ordinary $\Ext$ groups will be possible only for finitely generated graded $\D$-modules, for which $\grExt$ and $\Ext$ coincide.

\begin{thm}\label{gradedgennady}
Let $R = k[x_1, \ldots, x_n]$ and $\D = \D(R,k)$, and let $M$ be a graded $\D$-module.
\begin{enumerate}[(a)]
\item For all $i \geq 0$, $H^i_{dR}(\DD(M)) \cong \grExt^i_{\D}(M, E)$ as $k$-spaces.
\item If $M$ is finitely generated as a $\D$-module and has finite-dimensional de Rham cohomology (for instance, if $M$ is holonomic), then $\dim_k(H^{n-i}_{dR}(M)) = \dim_k(\Ext^i_{\D}(M,E))$ for all $i \geq 0$.
\end{enumerate}
\end{thm}

\begin{proof}
We note first that part (b) follows immediately from part (a): if $M$ is finitely generated as a $\D$-module, we have $\grExt^i_{\D}(M, E) = \Ext^i_{\D}(M,E)$ \cite[Corollary 2.4.7]{BookMethodsGradedRings}, and if $M$ has finite-dimensional de Rham cohomology, then $\dim_k(H^{n-i}_{dR}(M)) = \dim_k(H^i_{dR}(\DD(M)))$ for all $i \geq 0$ by Theorem \ref{dRdualsg}.

To prove part (a), we consider first the case $i = 0$. By definition, $H^0_{dR}(\DD(M))$ is the kernel of the $k$-linear map $\DD(M) \rightarrow \DD(M) \oplus \cdots \oplus \DD(M)$ defined by $\varphi \mapsto (\partial_1 \cdot \varphi, \ldots, \partial_n \cdot \varphi)$. That is, we have
\[
H^0_{dR}(\DD(M)) = \{\varphi \in \grHom_R(M, \E) \mid \partial_i \cdot \varphi = 0 \, \, \text{for all} \, \, i\}.
\]
By \eqref{hottastructure}, we have $(\partial_i \cdot \varphi)(m) = \partial_i \cdot \varphi(m) - \varphi(\partial_i \cdot m)$ for all $m \in M$ and $\varphi \in \grHom_R(M, \E)$. It follows that $\varphi \in H^0_{dR}(\DD(M))$ if and only if $\partial_i \cdot \varphi(m) = \varphi(\partial_i \cdot m)$ for all $m \in M$ and all $i$, that is, if and only if $\varphi$ is $\D$-linear. Therefore $H^0_{dR}(\DD(M)) \cong \grHom_{\D}(M, \E)$.

Since the functors $\{\grExt^i_{\D}(-, E)\}$ are the right derived functors of a left exact functor, they form a (contravariant) \emph{universal $\delta$-functor} as in \cite[pp. 205--206]{HartAG} from the category of graded $\D$-modules to the category of $k$-spaces. It suffices \cite[Corollary III.1.4]{HartAG} to prove that the functors $\{H^i_{dR}(\DD(-))\}$ also form a universal delta-functor, since they coincide for $i=0$. Short exact sequences of $\D$-modules (graded or otherwise) give rise to short exact sequences of de Rham complexes and therefore to long exact sequences of de Rham cohomology spaces; since the graded dual functor $\DD$ is exact, this implies that $\{H^i_{dR}(\DD(-))\}$ form a contravariant delta-functor. 

To show this $\delta$-functor is universal, we need only show that every $H^i_{dR}(\DD(-))$ is coeffaceable \cite[Theorem III.1.3A]{HartAG}; since the category of graded $\D$-modules has enough projective objects, it is enough to see that $H^i_{dR}(\DD(P)) = 0$ for all $i > 0$ and all projective graded $\D$-modules $P$. By Proposition \ref{de Rham as ext}, we have $H^i_{dR}(\DD(P)) = \Ext^i_{\D}(R, \DD(P))$, which is isomorphic to $\grExt^i_{\D}(R, \DD(P))$ since $R$ is a finitely generated graded $\D$-module. A projective object $P$ in the category of graded $\D$-modules is simply a projective $\D$-module that is graded \cite[Corollary 2.3.2, Remark 2.3.3]{BookMethodsGradedRings}, and any such object is a (graded) direct summand of a graded \emph{free} $\D$-module. Since de Rham cohomology commutes with direct sums, we may reduce the proof to the case where $P$ is graded free and hence further to the case $P = \D$.

To prove that $\grExt^i_{\D}(R, \DD(\D)) = 0$ for all $i > 0$, we use the \emph{change-of-rings spectral sequence} for $\grExt$:
\[
E_2^{p,q} = \grExt^p_{\D}(L, \grExt^q_R(\D, N)) \Rightarrow \grExt^{p+q}_R(L, N)
\]
for all graded $\D$-modules $L$ and graded $R$-modules $N$ (the ungraded version is \cite[Theorem 10.75]{rotman}, but since the category of graded $\D$-modules has enough projective and injective objects, there is also a graded version). Taking $L = R$ and $N = \E$, we get
\[
E_2^{p,q} = \grExt^p_{\D}(R, \grExt^q_R(\D, \E)) \Rightarrow \grExt^{p+q}_R(R, \E).
\]
The abutment is zero for $p + q > 0$ since $\E$ is injective as a graded $R$-module, and the $E_2^{p,q}$-term is zero for $q > 0$ for the same reason. Therefore the spectral sequence degenerates at $E_2$. For $q = 0$, we have $E_2^{p,0} = \grExt^p_{\D}(R, \grHom_R(\D, \E)) = \grExt^p_{\D}(R, \DD(\D)) = 0$ for all $p>0$, completing the proof.
\end{proof}

\section{A remark on $E$}\label{noninjective}
In this final section, we observe that $\E$ is not an injective object in the category of graded holonomic $\D$-modules. 

\begin{example}\label{Enoninjective}
Let $R = k[x]$ (or $k[[x]]$), let $d = \frac{d}{dx} \in \D$, and consider the quotients of $\D$ by the (left) ideals $\D \cdot x$, $\D \cdot xd$, and $\D \cdot d$. These quotients fit into a short exact sequence
\[
0 \rightarrow \D/(\D \cdot x) \xrightarrow{\cdot d} \D/(\D \cdot xd) \rightarrow \D/(\D \cdot d) \rightarrow 0
\]
of (left) $\D$-modules, where the map $\D/(\D \cdot x) \xrightarrow{\cdot d} \D/(\D \cdot xd)$ is \emph{right} multiplication by $d$. We have $\D/(\D \cdot x) \cong \E$ and $\D/(\D \cdot d) \cong R$ as $\D$-modules, so the corresponding long exact sequence in de Rham cohomology takes the form
\[
0 \rightarrow H^0_{dR}(\E) \rightarrow H^0_{dR}(\D/(\D \cdot xd)) \rightarrow H^0_{dR}(R) \rightarrow H^1_{dR}(\E) \rightarrow H^1_{dR}(\D/(\D \cdot xd)) \rightarrow H^1_{dR}(R) \rightarrow 0.
\]
By Example \ref{poincare}, the leftmost and rightmost terms are $0$, from which it follows that $H^0_{dR}(\D/(\D \cdot xd))$ and $H^1_{dR}(\D/(\D \cdot xd))$ are either both zero or both isomorphic to $k$. Consider 
\[
H^1_{dR}(\D/(\D \cdot xd)) = \frac{\D/(\D \cdot xd)}{d(\D/(\D \cdot xd))}.
\]
On the one hand, since $dx - xd = 1$ in $\D$, we have $\overline{dx} = \overline{1}$ in $\D/(\D \cdot xd)$, where the overline denotes the class of an element of $\D$ in the quotient. If we write $\overline{\overline{dx}}$ for the class of $\overline{dx}$ modulo $d(\D/(\D \cdot xd))$, we therefore also have $\overline{\overline{dx}} = \overline{\overline{1}}$. On the other hand, $d(\overline{x}) \in d(\D/(\D \cdot xd))$, so $\overline{d(\overline{x})} = \overline{\overline{0}}$ in $(\D/(\D \cdot xd))/d(\D/(\D \cdot xd))$. Clearly $\overline{d(\overline{x})} = \overline{\overline{dx}}$, so $\overline{\overline{0}} = \overline{\overline{1}}$, from which it follows that $H^1_{dR}(\D/(\D \cdot xd)) = 0$ (and therefore $H^0_{dR}(\D/(\D \cdot xd)) = 0$ as well).
\end{example}

In \cite[Corollary 2.10]{lsw}, it is proved that $\E$ is an injective object in the category of graded $\F$-finite $\F$-modules in characteristic $p>0$. This is rather surprising since, according to \cite[Example 4.8]{ma}, $\E$ is not an injective object in the category of $\F$-modules or the category of $\F$-finite $\F$-modules. Since $\F$-finite $\F$-modules (in characteristic $p$) are generally considered as counterparts of holonomic $\D$-modules (in characteristic zero), it is natural to ask if $\E$ is also an injective object in the category of graded holonomic $\D$-modules. Example \ref{Enoninjective} implies that this is {\it not} the case, which we state in the following proposition.

\begin{prop}\label{lswcounter1}
\begin{enumerate}[(a)]
\item If $R=k[[x]]$, then $\E$ is not an injective object in the category of holonomic $\D$-modules (\emph{a fortiori}, in the category of $\D$-modules).
\item If $R = k[x]$, then $\E$ is not an injective object in the category of graded holonomic $\D$-modules. 
\end{enumerate}
\end{prop}

\begin{proof} 
Let $M$ denote the $\D$-module $\D/(\D \cdot xd)$ of Example \ref{Enoninjective}. In both cases, Example \ref{Enoninjective} describes a short exact sequence $0 \rightarrow \E \rightarrow M \rightarrow R \rightarrow 0$ of $\D$-modules such that $H^1_{dR}(M) = 0$. Note that in both cases, this is an exact sequence of \emph{holonomic} $\D$-modules. 

In the case $R=k[[x]]$, Theorem \ref{dRdim and surjective map} implies that there does not exist a surjective homomorphism $M \rightarrow E$ of $\D$-modules, and hence the sequence cannot split. Hence $\E$ is not an injective object in the category of holonomic $\D$-modules, proving part (a).

In the case $R = k[x]$, since $x$, $d$, and $xd$ are homogeneous elements of $\D$, the objects in the short exact sequence are graded holonomic $\D$-modules. The map $\E \xrightarrow{\cdot d} M$ is homogeneous of degree $-1$. By a degree shift, we obtain a short exact sequence
\[
0 \rightarrow \E \rightarrow M(-1) \rightarrow R \rightarrow 0
\]
of graded holonomic $\D$-modules. By Theorem \ref{gradedhart}, there does not exist a surjective homomorphism $M \rightarrow E$ of $\D$-modules, and hence the sequence cannot split. Hence $\E$ is not an injective object in the category of graded holonomic $\D$-modules, proving part (b).
\end{proof}

\begin{remark}
Let $R$ be either $k[x]$ or $k[[x]]$. Theorems \ref{gradedgennady} and \ref{gennady} imply that $\dim_k\Ext^1_{\D}(R,\E)=\dim_kH^0_{dR}(R)=1$ and hence $\Ext^1_{\D}(R,\E)$ can be generated by the nontrivial extension $0 \rightarrow \E \rightarrow M \rightarrow R \rightarrow 0$. Since both $\E$ and $R$ are holonomic (and graded holonomic when $R=k[x]$), if $\Ext^1$ is defined in the category of holonomic $\D$-modules (and also in the category of graded holonomic $\D$-modules when $R=k[x]$, respectively) using Yoneda's characterization of $\Ext^1$, we have $\Ext^1_{\D,\mathrm{hol}}(R,\E)\cong k$ (and $\Ext^1_{\D,\mathrm{graded\ hol}}(R,\E)\cong k$, respectively).
\end{remark}

\begin{remark}\label{lswcounter2} 
Let $R=k[x_1,\dots,x_n]$ where $k$ is a field of characteristic $p>0$ and $\fm=(x_1,\dots,x_n)$. In \cite[Theorem 2.9]{lsw}, Lyubeznik, Singh, and Walther prove that each nonzero {\it graded} $\F$-finite $\F$-module $M$ admits a graded $\F$-finite submodule $N$ such that $M/N$ is supported in $\fm$ {\it and} $N$ does not admit any composition factor whose support is contained in $\{\fm\}$. Example \ref{Enoninjective} shows that the analogue of this result for graded holonomic $\D$-modules over polynomial rings in characteristic zero does {\it not} hold. 
\end{remark} 

\bibliographystyle{plain}

\bibliography{gradeddual}

\end{document}